\documentclass{amsart} 

\usepackage{amsmath,amssymb,amsthm,graphicx,subfig,color,comment,mathtools,microtype,tikz-cd,caption,wrapfig}
\usepackage[abs]{overpic}		
\usepackage[bookmarks=true,pdfencoding=auto,hidelinks]{hyperref}
\usepackage[nameinlink,nosort]{cleveref}
\let\cref\Cref
\newcommand{\myemail}[1]{\href{mailto:#1}{#1}}
\newcommand{\R}{\mathbb{R}}
\newcommand{\Z}{\mathbb{Z}}

\newcommand{\K}{\mathcal{K}}

\DeclareMathOperator{\rk}{rk}

\newtheoremstyle{thm}{}{}{\itshape}{}{\bfseries}{}{ }{} 
\newtheoremstyle{definition}{}{}{}{}{\bfseries}{}{ }{} 

\theoremstyle{thm}
\newtheorem{Theorem}{Theorem}[section]
\newtheorem{theorem}[Theorem]{Theorem}
\newtheorem{lemma}[Theorem]{Lemma}
\newtheorem{proposition}[Theorem]{Proposition}

\newtheorem*{theorem*}{Theorem}
\newtheorem{conjecture}[Theorem]{Conjecture}

\theoremstyle{definition}
\newtheorem{definition}[Theorem]{Definition}
\newtheorem{rem}[Theorem]{Remark}

\begingroup\expandafter\expandafter\expandafter\endgroup
\expandafter\ifx\csname pdfsuppresswarningpagegroup\endcsname\relax
\else
\fi

\begin{document}


\title[On unknotting fibered positive knots and braids]{On unknotting fibered positive knots and braids}

\author[M.\ Kegel]{Marc Kegel}
\address{Humboldt-Universit\"at zu Berlin, Rudower Chaussee 25, 12489 Berlin, Germany \newline \indent Universität Heidelberg, Im Neuenheimer Feld 205, 69120 Heidelberg, Germany}
\email{\myemail{kegemarc@hu-berlin.de}, \myemail{kegelmarc87@gmail.com}}

\author[L.\ Lewark]{Lukas Lewark}
\address{ETH Z\"urich, R\"amistrasse 101, 8092 Z\"urich, Switzerland}
\email{\myemail{lukas.lewark@math.ethz.ch}}

\author[N.\ Manikandan]{Naageswaran Manikandan}
\address{Humboldt-Universit\"at zu Berlin, Rudower Chaussee 25, 12489 Berlin, Germany}
\email{\myemail{naageswaran.manikandan@hu-berlin.de}}

\author[F.\ Misev]{Filip Misev}
\address{Universit\"at Regensburg, 93040 Regensburg, Germany}
\email{\myemail{filip.misev@mathematik.uni-regensburg.de}}

\author[L.\ Mousseau]{Leo Mousseau}
\address{Humboldt-Universit\"at zu Berlin, Rudower Chaussee 25, 12489 Berlin, Germany}
\email{\myemail{leo.mousseau@t-online.de}}

\author[M.\ Silvero]{Marithania Silvero}
\address{Universidad de Sevilla, Avda. Reina Mercedes, 41012 Seville, Spain}
\email{\myemail{marithania@us.es}}

\pdfstringdefDisableCommands{%
  \def\unskip{}%
}

\hypersetup{pdfauthor={\authors},pdftitle={\shorttitle}}




\begin{abstract}
The unknotting number $u$ and the genus $g$ of braid positive knots are equal, as shown by Rudolph.
We prove the stronger statement that any positive braid diagram of a genus $g$ knot
contains $g$ crossings, such that changing them produces a diagram of the trivial knot.
Then, we turn to unknotting the more general class of fibered positive knots,
for which $u = g$ was conjectured by Stoimenow.
We prove that the known ways to unknot braid positive knots do not generalize to fibered positive knots.
Namely, we prove that
there are fibered positive knots that cannot be unknotted optimally along fibered positive knots;
there are fibered positive knots that do not arise as trefoil plumbings;
and there are positive diagrams of fibered positive knots of genus $g$ that do not 
contain $g$ crossings, such that changing them produces a diagram of the trivial knot.
In fact, we conjecture that one of our examples is a counterexample to Stoimenow's conjecture.
\end{abstract}

\makeatletter
\@namedef{subjclassname@2020}{%
  \textup{2020} Mathematics Subject Classification}
\makeatother


\maketitle

\section{Introduction}
\begin{wrapfigure}[16]{r}{58mm}
 	\raggedleft
 	\raisebox{0mm}[.9\height]{
\includegraphics[width=57mm]{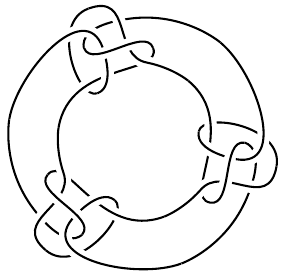}}
\captionsetup{width=.98\linewidth,justification=centering}
 	\caption{A potential counterexample~$\K$ to \cref{conj:pos_fib}.}
  \label{fig:counterexample} 
\end{wrapfigure}

This paper is inspired by the following conjecture posed by Stoimenow.

\begin{conjecture}\cite{Stoimenow03}\label{conj:pos_fib}
    If $K$ is a positive and fibered knot, then its unknotting number $u(K)$ equals its 3-genus $g(K)$.
\end{conjecture}

Note that the inequality $u \geq g$ holds for all strongly quasipositive knots~\cite{Rudolph93},
so the open part of the conjecture is to show $u \leq g$.
While discussing the conjecture, we refer to \cref{fig:props} for context. 
We will begin by examining the unknotting process of knots formed as the closure of positive braids, referred to as \emph{braid positive knots}. 
These knots are both fibered~\cite{Stallings78} and positive, and they satisfy $u = g$, i.e.~the conjecture holds for them. An elegant proof of $u \leq g$ for braid positive knots is due to Rudolph~\cite{Rudolph83}. That proof is by induction, with the following induction step:
every non-trivial braid positive knot $K$ is related by a crossing change to another braid positive knot $J$ such that $g(J) = g(K) - 1$. In other words, braid positive knots may be unknotted optimally along braid positive knots.
Alternatively, the inequality $u(K) \leq g(K)$ for braid positive knots $K$
may be deduced from the fact that such knots are \emph{positive trefoil plumbings}~\cite{BaaderDehornoy16}, i.e.~the fiber surface of $K$ arises from a disk by finitely many successive plumbings of the fiber surface of the positive trefoil knot. The inequality $u(K) \leq g(K)$ holds for all positive trefoil plumbings $K$, since the genus is additive under knot plumbings and the effect of deplumbing a trefoil from a knot can be achieved by a crossing change (see \cref{fig:plumbing}).

Now, we present a third proof of $u(K) \leq g(K)$ for braid positive knots $K$, as a corollary of the following theorem, which we prove in \cref{sec:pos_braids}.
\begin{theorem}\label{thm:unknotting_sequence}
Let $D$ be a knot diagram arising as the closure of a positive braid,
and let $g$ be the 3-genus of the represented knot.
Then there exists a set of $g$ crossings in~$D$ such that changing
those crossings produces a diagram of the unknot.
\end{theorem}
So far, we have discussed three different proofs that $u \leq g$ holds for braid positive knots.
The following three theorems (proven respectively in \cref{sec:unknotting_within,sec:trefoil_plumbings,sec:counterexample}) now demonstrate
that each one of those proofs fails to generalize to fibered positive knots.
Namely, fibered positive knots may not be unknotted optimally along fibered positive knots,
they need not be positive trefoil plumbings, and they may admit positive diagrams
that cannot be unknotted by performing $g$ many crossing changes.%
\begin{figure}[tb]
\[
\begin{tikzcd}[column sep={6.5em,between origins},row sep=small]
\fbox{\parbox{27mm}{\centering Braid positive\\ knot}} \ar[Rightarrow,rr]\ar[Rightarrow,dd] &&
\fbox{\parbox{27mm}{\centering Fibered positive\\ knot}} \ar[Rightarrow,"\text{\cite{Rud_Pos_is_SQP}}",rr]\ar[Rightarrow,"\text{\cite{Cromwell_homogeneous}}",dd]\ar[Rightarrow,dotted,ld,"?"] &&
\fbox{\parbox{27mm}{\centering Fibered strongly quasipositive knot}}\ar[Rightarrow,dd,"\text{\cite{Hedden_tightSQP}}", shift left=2] \\
& \fbox{\parbox{10mm}{\rule{0pt}{2ex}\centering $u = g$}} \\
\fbox{\parbox{27mm}{\centering Positive trefoil\\ plumbing}} \ar[Rightarrow,rr]\ar[Rightarrow,ru] &&
\fbox{\parbox{27mm}{\centering Positive Hopf\\ plumbing}} \ar[Rightarrow,rr] &&
\fbox{\parbox{27mm}{\centering Positive Hopf\\ plumbing \& deplumbing}} \ar[Rightarrow,uu,"\text{\cite{Ru_plumbing_QP}}", shift left=2]
\end{tikzcd}
\]
\caption{Implications between some properties of knots.
No further implications hold among these properties, with the only potential exception of \cref{conj:pos_fib}, drawn dotted.
For example, see \cite{MR0859157} for a fibered strongly quasipositive knot that is not a positive Hopf plumbing.}
\label{fig:props}
\end{figure}
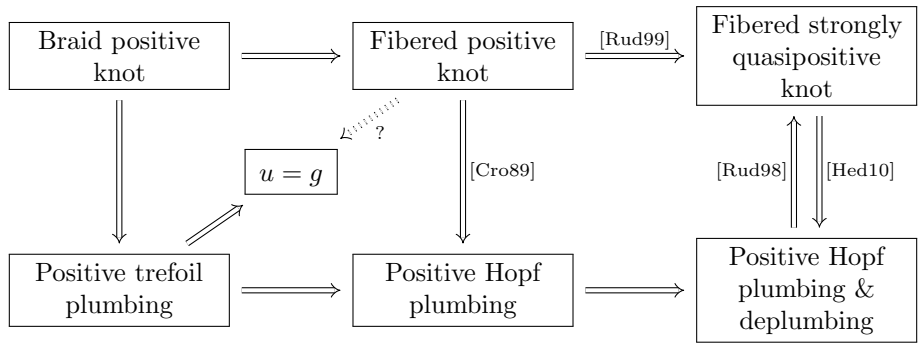%
\begin{theorem}
\label{thm:unknotting}
Let $K$ be the fibered positive knot $K11n183$ shown in \cref{fig:11n183},
which satisfies $u(K) = g(K) = 3$.
Then there is no crossing change relating $K$ with a fibered positive knot $J$ with $u(J) = 2$.
\end{theorem}
\begin{theorem}
\label{thm:plumbing}
The fiber surface of $K11n183$ is not a positive trefoil plumbing.
\end{theorem}
\begin{theorem}
\label{thm:counterexample}
Let $D$ be the positive knot diagram shown in \cref{fig:counterexample}.
Then $D$ represents a fibered positive knot $\K$ of genus $7$.
However, the least number of crossings of $D$
one needs to change to produce a diagram of the unknot equals $9$.
\end{theorem}
In fact, we make the following conjecture.
\begin{conjecture}\label{conj:counterexample}
    The fibered positive knot $\K$ shown in \cref{fig:counterexample} has unknotting number $9$ and is thus a counterexample to \cref{conj:pos_fib}. 
\end{conjecture}

In \cref{sec:counterexample}, we shall see that, however, no known lower bound for the unknotting number appears to be strong enough to settle \cref{conj:counterexample}.
We also present an infinite family of potential counterexamples for \cref{conj:pos_fib},
generalizing the knot~$\K$.

\subsection*{Acknowledgments} 
This work started at the \emph{Geometry and Topology Section} of the DMV (\emph{German Mathematical Society}) Annual Meeting 2022 in Berlin. 
M.K.\ is partially supported by the SFB/TRR 191 \textit{Symplectic Structures in Geometry, Algebra and Dynamics}, funded by the DFG (Projektnummer 281071066 - TRR 191).
L.L.~gratefully acknowledges support by the DFG (\emph{German Research Foundation}) via the Emmy Noether Programme, project no. 412851057.
N.M.\ is funded by the DFG under Germany's Excellence Strategy -- The Berlin Mathematics Research Center MATH+ (EXC-2046/1, project ID: 390685689).
M.S.\ was partially supported by Spanish Research Project PID2020-117971GB-C21, by IJC2019-040519-I, funded by MCIN/AEI/10.13039/501100011033 and by SI3/PJI/2021-00505 (Com. Madrid).
The authors wish to thank Sebastian Baader and Peter Feller for stimulating conversations.

\section{Unknotting of positive braids}\label{sec:pos_braids}

This section is devoted to the proof of \cref{thm:unknotting_sequence}, stating that knots given by a positive braid diagram may be optimally unknotted within that diagram. An example is shown in \cref{fig:torusunknotting} from which the general argument can be deduced easily. 
\begin{figure}[b]
\includegraphics{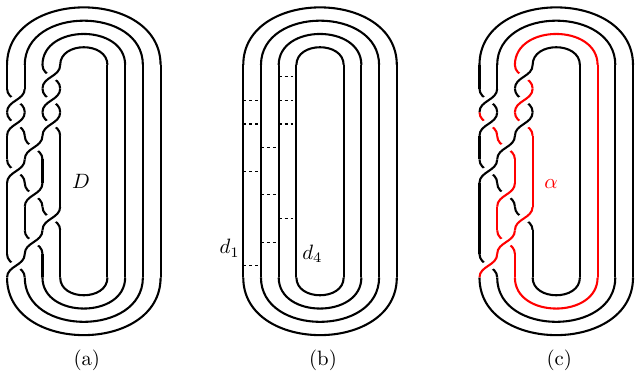}
\caption{(a) Example of a diagram $D$ representing a braid positive knot $K$. (b) Seifert circles obtained when applying Seifert's algorithm to $D$. (c) A sub-arc $\alpha\subset D$ with a single double point.} \label{fig:braid}
\end{figure}

\begin{proof}[Proof of \cref{thm:unknotting_sequence}]
Let $K$ be a braid positive knot and $D$ a positive braid diagram representing $K$; see \cref{fig:braid}(a) for an example. We may think of $D$ as an oriented immersed self-transverse circle in $\R^2$. The positivity of the crossings is then enough to determine the over-/undercrossing information at each double point. Note that the Seifert circles $d_1,d_2,\ldots,d_n$ that arise when Seifert's algorithm is applied to $D$ are nested, as in \cref{fig:braid}(b). We choose the numbering in such a way that the disk bounded by $d_i$ contains $d_{i+1}$, for all $i\in\{1,2,\ldots, n-1\}$. In particular, $d_1$ is the outermost and $d_n$ is the innermost Seifert circle. Note that $d_i\cap D$ consists of finitely many sub-arcs of $d_i$, separated by the crossings that connect $d_i$ with the adjacent Seifert circles.

We proceed by induction on the number of crossings of $D$. If $D$ has no crossings, or, more generally, if $K$ is the unknot, we are done. Hence, assume that $K$ is non-trivial. Let $\alpha\subset D$ be an arc such that $\alpha\cap d_1$ has exactly two connected components, each containing one of the endpoints of $\alpha$. In other words, $\alpha$ is the projection of an arc in $K$, starts at a point of $d_1$, travels along $D$ in the given orientation, and stops the first time it comes back to $d_1$. An example of such an arc is shown in \cref{fig:braid}(c). Since $D$ is a positive braid diagram, every passage of $\alpha$ from a Seifert circle $d_i$ to $d_{i+1}$ through a double point is an overpass and every passage of $\alpha$ from $d_{i+1}$ to $d_i$ is an underpass. This implies that $\alpha$ contains the same number $\ell\geq 1$ of over- and underpasses. Note that there does indeed exist such a path $\alpha$ because $K$ is non-trivial; in particular, there are at least two Seifert circles and at least one crossing connecting $d_1$ to $d_2$.

First, consider the case in which every double point of $D$ is traversed at most once by $\alpha$. Changing the crossings that correspond to the $\ell$ underpasses on $\alpha$, turns $\alpha$ into an arc that only contains overpasses. A planar isotopy can then be applied to remove all $2\ell$ crossings encountered by $\alpha$, preserving the rest of the diagram. We thus obtain a new positive braid diagram $D'$ with $2\ell$ fewer crossings than $D$. The knot $K'$ represented by $D'$ differs from $K$ by $\ell$ crossing changes. When applying Seifert's algorithm to $D'$, we still obtain the same number $n$ of Seifert circles; therefore $g(K')+\ell = g(K)$. By induction, we can find $g(K')$ crossing changes within $D'$ which unknot $K'$. These $g(K')$ crossings did not move under the isotopy from $D$ to~$D'$, so we find them in the original diagram, $D$. Together with the $\ell$ crossings at the underpasses of $\alpha$, we obtain $g(K')+\ell = g(K)$ crossings in $D$ such that changing these turns $D$ into a diagram of the unknot, as claimed. 
\begin{figure}[b]
\centering
\includegraphics[width=\textwidth]{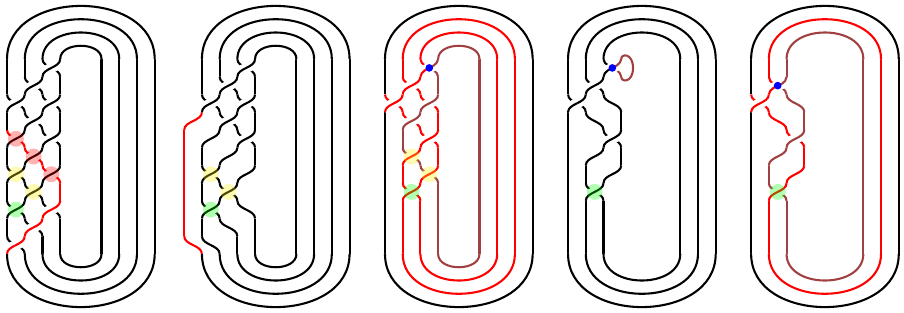}
\caption{The algorithm from the proof of \cref{thm:unknotting_sequence}, applied to the torus knot $T(4,5)$. Crossings to be changed are highlighted in red, yellow, and green (they are changed in this order); arcs $\alpha$ in red, loops $\alpha_0$ in dark red, double points of loops $\alpha_0$ marked with a blue dot.} 
\label{fig:torusunknotting}
\end{figure}

Now consider the case in which $\alpha$ has at least one double point. We replace $\alpha$ by a sub-arc 
$\alpha_0\subset\alpha$ that starts and ends at a double point $c$, that is, $\alpha_0$ shall not self-intersect, whilst it may very well pass through crossings of $D$. Such a loop can easily be found by walking along $\alpha$ until the first self-crossing of that walk occurs. Now we proceed as in the previous case: perform $\ell_0\geq 0$ crossing changes in $D$, transforming the underpasses in $\alpha_0$ (away from $c$, if there are any) into overpasses, apply a planar isotopy from $\alpha_0$ to a small loop that does not contain any double points of $D$ and finally perform a type~(I) Reidemeister move removing the crossing~$c$. This produces a positive braid diagram $D'$ of a knot $K'$, having $2\ell_0+1$ fewer crossings than $D$, and $K'$ is related to $K$ by $\ell_0$ crossing changes. The number of Seifert circles in $D'$ is $n-1$, and therefore $g(K')+\ell_0=g(K)$. Again by induction, we can find $g(K')$ crossing changes in $D'$ turning $K'$ into the unknot; these $g(K')$ crossings did not move under the isotopy from $D$ to $D'$, so we find them within the original diagram $D$. Together with the $\ell_0$ crossings at the underpasses of $\alpha_0$, we found a set of $g(K')+\ell_0=g(K)$ crossings in $D$ such that changing them turns $D$ into a diagram of the unknot. 
\end{proof}

\section{Obstructing the unknotting of fibered positive knots along fibered positive knots}\label{sec:unknotting_within}

\begin{figure}[b] 
 	\centering
 	\includegraphics[width=\textwidth]{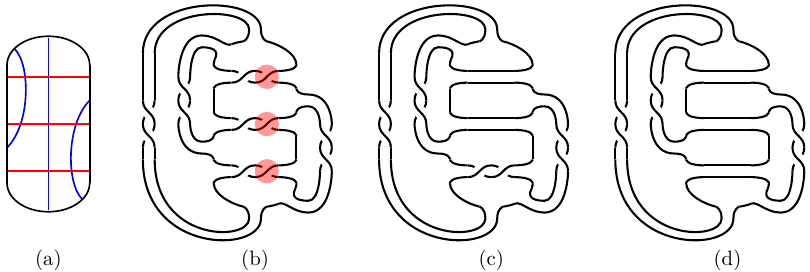}
 	\caption{(a) Disc with prescribed arcs for plumbing positive Hopf bands from the top, in red, and from below, in blue. (b) Resulting positive diagram of $K11n183$. (c) The result of changing the top two of the three marked crossings is the twist knot $K6a3$, which is not fibered. (d) Diagram of the unknot obtained by changing all three marked crossings.}
  \label{fig:11n183} 
\end{figure}

This section contains the proof of~\cref{thm:unknotting}, stating that there is no crossing change from the (positive, fibered) knot $K11n183$ with unknotting number~$3$ to any fibered positive knot with unknotting number~$2$.
\cref{fig:11n183}(a) and (b) show a description of the protagonist of this section, the knot $K11n183$, as a plumbing of positive Hopf bands. \cref{fig:11n183}(b)--(d) moreover indicate how to unknot $K11n183$ with 3 crossing changes.

The Gordian distance of two knots is the minimum number of crossing changes required to transform one knot into the other. In the proof of \cref{thm:unknotting} we make use of a new lower bound for Gordian distances in terms of the homology of branched coverings, which we give in \cref{prop:pid}.
As preparation, let us make two definitions and prove one lemma in a rather general algebraic context.
\begin{definition}\label{def:n}
Let $R$ be a PID, $q\in R$ a prime element and $e \geq 1$ an integer.
For a matrix $A$ over~$R$, denote by $n_{q,e}(A)$ the maximal integer
such that the cokernel of~$A$ admits a summand $R/(q^{f_1}) \oplus \dots \oplus R/(q^{f_{n_{q,e}(A)}})$
with~$f_1, \dots, f_{n_{q,e}(A)} \geq e$.
\end{definition}
\begin{definition}\label{def:rank}
The \emph{rank} $\rk_S A$
of an $m\times m$ matrix $A$ over a commutative ring~$S$
is the smallest number of generators for the submodule of $S^m$ generated
by the columns of $A$.
\end{definition}
\begin{lemma}\label{lemma:algebra}
Let $A$ be an $m\times m$ matrix over a PID $R$, $q\in R$ a prime element and $e \geq 1$ an integer.
Let $S$ be the ring $R / (q^e)$. Then $n_{q,e}(A) = m - \rk_S A$.
\end{lemma}
\begin{proof}
We can simplify the situation by applying the Smith normal form. Because $R$ is a~PID, there are invertible matrices $U$ and $V$ such that the product $UAV$ is a diagonal matrix with entries $d_1, \ldots, d_m$, where each $d_i$ divides~$d_{i+1}$. It is important to note that the rank and the isomorphism type of the cokernel of $A$ remain unchanged when we multiply by $U$ and~$V$,
i.e.~$\rk_S A  = \rk_S UAV$ and $n_{q,e}(A) = n_{q,e}(UAV)$.
Hence $n_{q,e}(A)$ equals the number of $d_i$ divisible by~$q^e$. When we consider $UAV$ as a matrix over the ring~$S$, these entries $d_i$ become~$0$. This proves $\rk_S A \leq m - n_{q,e}(A)$.

Let us now show the reverse inequality. Let $\ell = m - n_{q,e}(A)$ and
for $i \in \{1,\ldots,\ell\}$, let $f_i$ be maximal such that $q^{f_i}$ divides $d_i$.
By assumption, we have ${0\leq f_i < e}$.
Note that the $S$-module generated by the columns of $UAV$ is isomorphic to
$S / (q^{e - f_1}) \oplus \dots \oplus S / (q^{e - f_{\ell}})$.
Let a system of $a$ generators of this module be given, i.e.~a surjective $S$-homomorphism
\[
S^a \to S / (q^{e - f_1}) \oplus \dots \oplus S / (q^{e - f_{\ell}}).
\]
Composing with multiplication by $q^{e - f_i - 1}$ in the $i$-th component,
this yields a surjective homomorphism
$S^a \to (S / (q))^{\oplus \ell}$.
Since the target module is annihilated by~$q$, this induces a surjective
$S/(q)$-homomorphism $(S/(q))^a \to (S / (q))^{\oplus \ell}$.
But since $S/(q) \cong R_p/(q)$ is a field, we have $a\geq \ell$.
Thus we have established that any system of generators of the column space of $UAV$ over $S$ has at least $\ell$ 
generators, and so the desired inequality $\rk_S A\geq \ell = m - n_{q,e}(A)$ follows.
\end{proof}

We shall now apply our newly found algebraic wisdom to the homology of branched coverings.
Namely, for a knot $K$ and an integer $n \geq 2$, denote by $\Sigma_n(K)$ the $n$-fold branched covering of $S^3$ along $K$ (see e.g.~\cite[Section~8.E]{MR3156509} for an introduction). 
The action of the deck transformation group $\langle t \,|\, t^n\rangle$ on $\Sigma_n(K)$
makes the homology groups of $\Sigma_n(K)$ into modules over the ring
$\Z[t]/(t^n-1)$. Since $(1 + \dots + t^{n-1})$ acts trivially on homology,
one may also consider homology as $R_n$-modules, for $R_n \coloneqq \Z[t]/(1 + \dots + t^{n-1})$. To compute $H_1(\Sigma_n(K); R_n)$, we will use the following lemma.

\begin{lemma}[{\cite[Proposition~8.39]{MR3156509}}]\label{lem:presmatrix}
If $B$ is a Seifert matrix of a knot $K$, then $tB - B^{\top}$ is a presentation matrix 
of $H_1(\Sigma_n(K); R_n)$ over~$R_n$.
\end{lemma}
In other words, $H_1(\Sigma_n(K); R_n)$ is isomorphic to the cokernel of $tB - B^{\top}$.

\begin{proposition}\label{prop:pid}
Let the integer $p\geq 2$ be a prime such that $R_p$ is a~PID.
\footnote{This is the case if and only if $p \leq 19$ (see e.g.~\cite[Theorem~11.1]{MR1421575}).
We expect that an adapted version of \cref{prop:pid} holds for all primes $p$, using that $R_p$ is always a Dedekind domain. But since we will only use the case $p = 3$ in this paper, we do not pursue this further.}
Let $q\in R_p$ be a prime element and $e \geq 1$ an integer. For a knot $K$, denote by $n_{p,q,e}(K) \geq 0$
the maximal integer such that $H_1(\Sigma_p(K); R_p)$ admits a summand
$R_p/(q^{f_1}) \oplus \dots \oplus R_p/(q^{f_{n_{p,q,e}(K)}})$ with $f_1, \dots , f_{n_{p,q,e}(K)} \geq e$.
Then $|n_{p,q,e}(K) - n_{p,q,e}(J)|$ is a lower bound for the Gordian distance between the knots $K$ and $J$.
\end{proposition}
By \cref{lem:presmatrix},
$n_{p,q,e}(K)$ equals $n_{q,e}(tB - B^{\top})$ (see \cref{def:n}) for a Seifert matrix $B$ of $K$.
Let us first deduce \cref{thm:unknotting} from \cref{prop:pid}, and prove the proposition afterwards.

\begin{proof}[Proof of~\cref{thm:unknotting}]
Let us assume towards a contradiction that $J$ is a fibered positive knot with $u(J) = 2$,
obtained from $K11n183$ by a single crossing change.
It follows from~\cite[Corollary~5.1]{Cromwell_homogeneous} that the crossing number of a fibered positive knot is at most four times its genus. Thus $c(J) \leq 4\cdot g(J) \leq 4\cdot u(J) = 8$.
A glance at the knot table~\cite{KnotInfo} tells us that $J$ must be $T(2,3)^{\# 2}$ or $T(2,5)$.
So it just remains to prove that $K11n183$ has Gordian distance at least $2$ from both of those knots.
Using the presentation matrices given by \cref{lem:presmatrix}, and~\cite{sage} to compute their Smith normal forms, one finds:
\begin{align*}
H_1(\Sigma_3(K11n183); R_3)       & \ \cong\ (R_3/(4))^{\oplus 2}, \\
H_1(\Sigma_3(T(2,3)^{\#2}); R_3) & \ \cong\ (R_3/(2))^{\oplus 2}, \\
H_1(\Sigma_3(T(2,5)); R_3)       & \ \cong\ 0.
\end{align*}
Thus the values of $n_{3,2,2}$ on $K11n183$, $T(2,3)^{\#2}$ and $T(2,5)$ are $2$, $0$ and $0$, respectively. By \cref{prop:pid}, this concludes the proof of \cref{thm:unknotting}.
\end{proof}
\begin{proof}[Proof of \cref{prop:pid}]
Assume that $K$ and $J$ are related by a single crossing change.
It will suffice to show that $|n_{p,q,e}(K) - n_{p,q,e}(J)| \leq 1$.
The knots $K$ and $J$ admit respective Seifert matrices $B$ and $C$ such that
\[
C = \begin{pmatrix}
B & v     & 0 \\
w^{\top} & x & 1 \\
0 & 0     & \pm 1
\end{pmatrix}
\]
for some integer column vectors $v,w$ and for $x\in \Z$~\cite{Sa_ua}.
Consider the matrix $E$, which has the same size as matrix $C$ but with all entries being zero except for the lower right corner, which is $\mp 1$. Now, let $\widetilde{B} = C + E$. 
This matrix $\widetilde{B}$ is S-equivalent to the matrix $B$.
Define $B' = t\widetilde{B} - \widetilde{B}^{\top}$, $C' = tC - C^{\top}$, and $E' = tE - E^{\top}$, all considered as matrices over the ring $R_p$.
Note that by \cref{lem:presmatrix}, $B'$ and $C'$ serve as presentation matrices for $H_1(\Sigma_p(K); R_p)$ and $H_1(\Sigma_p(J); R_p)$, respectively. 
Furthermore, we have the relationship $C' = B' - E'$.
With $S = R_p/(q^e)$,
 the rank $\rk_S E'$ of the matrix $E'$ is equal to~$1$ (see \cref{def:rank}).
We now have
\[
n_{p,q,e}(J) = m - \rk_S C' \geq m - \rk_S B' - \rk_S E' = n_{p,q,e}(K) - 1,
\]
where the equalities are due to \cref{lemma:algebra},
and the inequality due to the subadditivity of $\rk_S$,
which follows quickly from its definition.
Hence $n_{p,q,e}(K) - n_{p,q,e}(J) \leq 1$. Switching the roles of $K$ and $J$
gives us $|n_{p,q,e}(J) - n_{p,q,e}(J)| \leq 1$, as desired.
\end{proof}
\begin{rem}
We are not aware of an alternative method to bound the Gordian distance of $K11n183$ and $T(2,3)^{\# 2}$.
For the Gordian distance of $K11n183$ and $T(2,5)$, on the other hand, one does not need the full strength of \cref{prop:pid}.
Indeed, by the universal coefficient theorem, $H_1(\Sigma_3(K11n183); \mathbb{F}_2) \cong \mathbb{F}_2^4$ and ${H_1(\Sigma_3(T(2,5));\mathbb{F}_2) = 0}$.
Now it suffices to use the following well-known lower bound for the Gordian distance of $K$ and $J$,
established by Wendt~\cite{wendt} (in the very same paper in which the unknotting number was first introduced): namely, for all primes $p,q\geq 2$,
\[
\frac{1}{p-1} \Big|
\dim_{\mathbb{F}_q}H_1(\Sigma_p(K); \mathbb{F}_q) -
\dim_{\mathbb{F}_q}H_1(\Sigma_p(J); \mathbb{F}_q)
\Big|
\]
is less than or equal to the Gordian distance of $K$ and $J$.

Yet another way to show that $K11n183$ and $T(2,5)$ have (algebraic) Gordian distance at least $2$,
is to use~\cite[Corollary~2.8]{Murakami_KnotMetrics}.
\end{rem}
\begin{rem}
Since the lower bound in \cref{prop:pid} (as well as Wendt's bound) only depends on the S-equivalence class of the Seifert matrix, it follows that they are actually lower bounds for the \emph{algebraic Gordian distance}.
As a consequence, there is not even an algebraic unknotting move transforming
any Seifert matrix of $K11n183$ into a Seifert matrix of a fibered positive knot $J$ with $u(J) = 2$.
\end{rem}

\section{Obstructing trefoil plumbings}\label{sec:trefoil_plumbings}

For context, we first remind the reader that every fibered link is related to the unknot via a sequence of Hopf plumbings and deplumbings~\cite{Harer_Hopf82, GirouxGoodman_plumbing, giroux_hopf, Goodmann_hopf}. Here, Hopf plumbing is the operation of gluing a (positive or negative) Hopf band onto a given Seifert surface $S$, where the gluing region is specified by a properly embedded interval on $S$ (see~\cref{fig:plumbing} (left arrow) for an illustration). The result of a Hopf plumbing operation is a new Seifert surface whose first Betti number is one more than that of $S$. Note that an isotopy of the plumbing arc (along properly embedded arcs in $S$) induces an isotopy of the resulting Seifert surface. Beware that the plumbing operation occurs locally in a neighbourhood of the plumbing interval; in particular, the plumbed Hopf band must not tangle with any part of the Seifert surface $S$.

\begin{figure}[b]
\centering
\includegraphics[width=\textwidth]{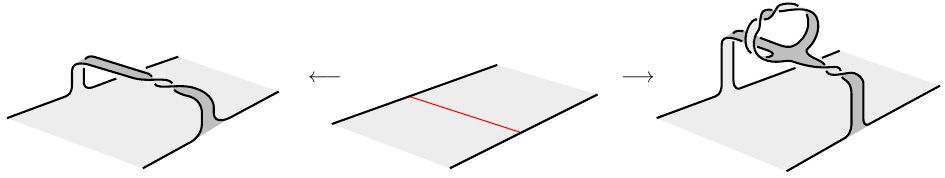}
\caption{Plumbing a positive Hopf band (left) or a positive trefoil (right) to a Seifert surface (middle). The plumbing operation is specified by an interval (shown in red), properly embedded in the Seifert surface, and occurs in a neighbourhood of that interval.}
\label{fig:plumbing}
\end{figure}

It is known that a given fibered knot has, up to isotopy, a unique Seifert surface of minimal genus, namely its fiber surface (see e.g.\ \cite[Proposition 2.19]{Rudolph_ComplexCurves} and~\cite[Lemma 5.1]{EdmondsLivingston_Fibered}). This allows us to make no distinction between fibered knots and their fiber surfaces. When we speak of a knot arising as a plumbing of Hopf bands (or trefoils), we mean that the knot is the boundary of a fiber surface obtained by plumbing Hopf bands (or trefoils).

In this section, we provide the proof of \cref{thm:plumbing}, which demonstrates that the knot $K11n183$ cannot be obtained by plumbing trefoils to an unknot. The trefoil plumbing operation is illustrated in \cref{fig:plumbing} (right arrow).

\begin{proposition}\label{prop:twotrefoils}
Let $K$ be a knot arising as the plumbing of two positive trefoils.
Then either $K$ is one of the three knots $T(2,3)^{\# 2}, T(2,5), K10n14$,
or the Levine--Tristram  signature evaluated at $e^{0.11\pi i}$ equals $-2$.
\end{proposition}
\begin{proof}
Given a fiber surface obtained as plumbing of two positive trefoil fiber surfaces, it is possible to find an associated Seifert matrix of the form: 
\begin{equation*}
A= \begin{pmatrix*}[r]
-1& 1& 0& 0\\   
0& -1& 0& 0\\
0& 0& -1& 1\\
a& b& 0& -1\\
\end{pmatrix*}
\end{equation*}
for integers $a, b$. Note that $a,b$ hold geometric significance: the second trefoil surface, attached to the first, consists of two Hopf bands (see \cref{fig:plumbing}). One Hopf band is distant from the fiber surface associated to the first trefoil, while the other intersects it in a thickened interval. The pair $(a, b)$ represents the algebraic intersection of this interval with the first trefoil surface's basis elements of the first homology group. Hence, either $(a, b) = (0, 0)$ or $a$ and~$b$ are coprime. Since the fiber surface associated to a trefoil is a punctured torus, the homology class of the plumbing interval already determines its isotopy class. Thus the pair~$(a,b)$ determines the isotopy class of the resulting plumbing surface (and the knot spanned by it).

Setting $d \coloneqq a^2 + ab + b^2 \geq 0$, one finds that
the Alexander polynomial obtained from $A$ equals
\[
\Delta_d(t) \coloneqq (t^{-1} - 1 + t)^2 + d\cdot (t^{-1} - 2 + t).
\]
One verifies that either one of the following cases holds, or $d \geq 7$:
\[
\begin{array}{l|l|l}
(a,b) & d & K \\[.5ex]\hline
\rule{0pt}{2.5ex}(0,0) & 0   & T(2,3)^{\# 2} \\
\rule{0pt}{2.5ex}(1,0),(0,1),(-1,1),(-1,0),(0,-1),(1,-1) & 1 & T(2,5) \\
\rule{0pt}{2.5ex}(1,1),(-1,2),(-2,1),(-1,-1),(1,-2),(2,-1) & 3 & K10n14 \\
\end{array}
\]
So it just remains to show that if $d \geq 7$, then $\sigma_{e^{0.11 \pi i}}(K) = -2$.
Let us assume $d\geq 7$. One calculates
\[
\Delta_d(t) = (\Delta_d(t) - \Delta_7(t)) + \Delta_7(t) = (d-7)\cdot (t^{-1} - 2 + t) + \Delta_7(t).
\]
Since the first summand is non-positive for $t = e^{0.11 \pi i}$, we have
\[
\Delta_d(e^{0.11\pi i}) \leq \Delta_7(e^{0.11\pi i}) < 0.
\]
Recall that $\sigma_{e^{0.11 \pi i}}(K)$ may be calculated as the signature
of the Hermitian matrix $B \coloneqq (1 - e^{0.11 \pi i})A + (1 - e^{-0.11 \pi i})A^{\top}$,
i.e.~$\sigma_{e^{0.11 \pi i}}(K) = n_+(B) - n_-(B)$ for $n_{\pm}$ the number of positive and negative eigenvalues of $B$, respectively.
Note that the top left $3\times 3$ submatrix of $B$ does not depend on $(a,b)$,
and one computes that it has one positive and two negative eigenvalues. It remains to prove that the fourth eigenvalue is also negative. Since $\det(B)=\Delta_d(e^{0.11\pi i})$ is negative, as we showed before, and $\operatorname{sign}(\det B) = (-1)^{n_-(B)}$, it follows that $n_+(B) = 1$ and $n_-(B) = 3$, as desired.
\end{proof}

\begin{proof}[Proof of \cref{thm:plumbing}]
A plumbing of $n$ trefoil fiber surfaces is a fiber surface, and thus genus-minimizing.
Let us assume towards a contradiction that $K11n183$ is a plumbing of positive trefoils.
Since the genus is additive under plumbing and $g(K11n183) = 3$, it follows that $K11n183$ is a plumbing of precisely three trefoils.
By \cref{prop:twotrefoils}, the knot $K11n183$ is a plumbing of a (positive) trefoil to one of the following knots~$K$: (1) $T(2,3)^{\# 2}$, (2) $T(2,5)$, (3) $K10n14$, or (4) a knot $K$ with $\sigma_{e^{0.11 \pi i}}(K) = -2$. Let us rule out these cases one by one. Note that it is enough to show that $K$ cannot be obtained from $K11n183$ by a positive-to-negative crossing change.

In cases (1) and (2), \cref{thm:unknotting} applies and gives the result. For case (3) we proceed similarly as in the proof of that theorem and compute
\[
H_1(\Sigma_3(K10n14); R_3) \cong R_3/(5),
\]
so \cref{prop:pid} implies that the Gordian distance of $K11n183$ and $K10n14$ is at least~$2$.
For case (4), recall that a positive-to-negative crossing change does not decrease the Levine--Tristram  signature at a fixed point on the circle. But this contradicts
$\sigma_{e^{0.11 \pi i}}(K) = -2 < 0 = \sigma_{e^{0.11 \pi i}}(K11n183)$.
\end{proof}

\begin{rem}
There exists an obstruction for a knot $K$ being the plumbing of a positive trefoil and another knot $J$, expressed in terms of the Alexander polynomial $\Delta_K$ of $K$, as demonstrated by Hironaka~\cite{Hironaka}.
In our context, Hironaka's obstruction implies that if the knot $K$ arises from the unknot by successive plumbings of positive trefoils, then the leading coefficient of $\Delta_K$ and the evaluation $\Delta_K(0)$, both of which are $\pm 1$, have the same sign. Note that this condition does not depend on the choice of normalization of $\Delta_K$; it is equivalent to Seifert matrices $A$ associated to the fiber surface of $K$ satisfying $\det(A) = 1$.
However, one can see that this condition is in fact satisfied by all knots that arise from the unknot by the plumbing and deplumbing of positive Hopf bands, i.e.~by all fibered strongly quasipositive knots.
So Hironaka's criterion cannot be used to obstruct a fibered positive knot from being a positive trefoil plumbing.
\end{rem}

\section{Fibered positive knots with potentially large unknotting number}\label{sec:counterexample}
    \begin{figure}[b] 
 	\centering
 	\includegraphics[width=0.9\textwidth]{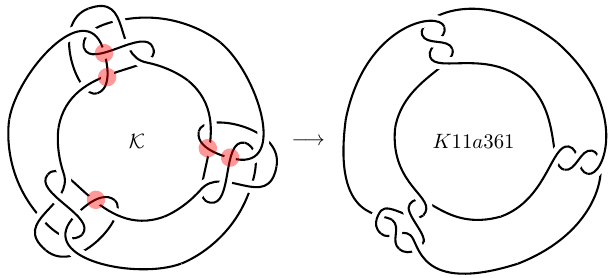}
  \caption{Applying five crossing changes to $\K$ to obtain $K11a361$.}
  \label{fig:11a361} 
 \end{figure}

Let us begin this section by proving \cref{thm:counterexample},
which claims that the knot $\K$ is a fibered positive knot of genus $7$,
but one cannot unknot it by 7 crossing changes in its diagram $D$ shown in \cref{fig:counterexample}.
\begin{proof}[Proof of \cref{thm:counterexample}]
Since $D$ is a positive diagram, $\K$ is a positive knot.
To confirm that $\K$ is fibered and compute its genus,
observe that $\K$ is a plumbing of 14 positive Hopf bands.
These statements about $\K$ are also 
the particular case $n=1$ of \cref{prop:counterexample_general}\eqref{item:fibpos} and \eqref{item:genus} below, proven below in greater detail.

To see that the minimal number of crossing changes needed to transform $D$ into a diagram of the unknot is $9$, we have checked all possible combinations of at most $8$ crossing changes in $D$ and observed that the Alexander polynomials of the resulting knots are non-trivial. The computations can be accessed at~\cite{data}.
An unknotting sequence with 9 crossings is easily found in $D$ (see also \cref{prop:counterexample_general}\eqref{item:uD} below).
\end{proof}

\begin{rem}\label{rem:types}
Since we cannot find 8 crossing changes that unknot $\K$,
we conjectured (see \cref{conj:counterexample}) that in fact $u(\K) = 9$, which would make $\K$ a counterexample to Stoimenow's \cref{conj:pos_fib}.
Let us now discuss why $u(\K) > 7$ is hard to show.
Although the literature is awash with lower bounds for the unknotting number,
most of them fall into one of the following four types:
\begin{itemize}
    \item [(a)] Bounds that are also lower bounds for the smooth 4-genus~$g_4$, for example, the bounds coming from the Levine--Tristram  signatures~\cite{levine-signature,tristram}, from 
the $\Upsilon$-invariant~\cite{MR3667589}, or from slice-torus invariant such as the $\tau$-invariant from knot Floer homology~\cite{tau} or the $s$-invariant from Khovanov homology~\cite{s_inv}.
    \item [(b)] Bounds that are also lower bounds for the \emph{algebraic unknotting number}~$u_a$, that is, the minimal number of crossing changes needed to transform a given knot into a knot with trivial Alexander polynomial. These are precisely the bounds determined by the S-equivalence class of the Seifert matrix, such as bounds coming from the Levine--Tristram  signatures, or coming from homology of cyclic coverings (such as \cref{prop:pid}), and further bounds in~\cite{Na_ua,Mu_ua,Fo_ua,Sa_ua,St004,BFII_ua,BFI_ua}.
    \item [(c)] Bounds that are also bounds for the \emph{proper rational unknotting number}~$u_q$, that is, the minimal number of proper rational tangle replacements needed to transform a given knot into the unknot. See for example~\cite{Lines,McCoy_rat_un,iltgen2021khovanov,arXiv:2203.09319,McCoy_Zentner}.
    \item [(d)] Bounds that are at most two~\cite{MR0808108,Mi_u1,OS_u1,GL_u1,Greene_u1,McCoy_u1}.
\end{itemize}
Note that the inequalities $g_4, u_a, u_q \leq u$ hold for all knots.
As we shall see below in \cref{prop:g4uauqK}, $g_4, u_a, u_q \leq 7$ for $\K$, so none of the bounds of type (a), (b), (c) may be used to establish $u(\K) > 7$; nor may one use bounds of type (d), of course.
One of the few lower bounds for $u$ not subsumed by (a)--(d)
is Owens' obstruction~\cite[Theorem~3]{Owens08} coming from Heegaard Floer homology,
which can, however, only be used to show $u(\K) > |\sigma(\K)/2| = 5$.
\end{rem}
\begin{proposition}\label{prop:g4uauqK}
The knot $\K$ shown in \cref{fig:counterexample} satisfies
$g_4(\K) = 7$, $u_a(\K) \leq 7$ and $u_q(\K) \leq 5$.
\end{proposition}
\begin{proof}
Since $\K$ is positive, $g_4(\K) = g(\K) = 7$ follows.
The statement $u_q(\K) \leq 5$ will be shown in greater generality in \cref{prop:counterexample_general} below. Let us now prove $u_a(\K) \leq 7$.
\cref{fig:11a361} shows a set of five crossing changes from $\K$ to the knot $K11a361$.
So it suffices to show that $K11a361$ has algebraic unknotting number $2$,
which we shall do now (note that $u_a(K11a361) = 2$ is stated in the online tables~\cite{knotorious,KnotInfo}, but no certificate is provided).

    Applying Seifert's algorithm to the diagram of $K11a361$ shown in \cref{fig:11a361} and appropriately choosing a basis of the first homology of the resulting surface yields the following Seifert matrix of $K11a361$: 
    \begin{equation*}
A=\begin{pmatrix*}[r]
-1& 1& 0& 0\\   
0& -1& 1& 0\\
0& 0& -3& 2\\
0& 0& 1& -3\\
\end{pmatrix*}
\end{equation*} Consider the base change given by the matrix $P$ defined below \begin{equation*}
P=\begin{pmatrix*}[r]
1& 1& 0& 0\\   
1& 0& 0& 0\\
-2& -1& -1& -1\\
-2& -1& -1& 0\\
\end{pmatrix*},
\end{equation*} and apply an algebraic unknotting move by adding $w^{\top} w$ where $w= \begin{pmatrix}
1 & 1 & -1 & -1
\end{pmatrix} $
to obtain the matrix  \begin{equation*}
B=PAP^{\top} + w^{\top}w=\begin{pmatrix*}[r]
0& 0& 0& 0\\   
1& 0& 0& 0\\
0& 1& -4& -3\\
0& 1& -2& -4\\
\end{pmatrix*}.
\end{equation*} Applying a reduction to $B$ to remove the first and the second rows/columns, and performing a new algebraic unknotting move by adding $v^{\top} v$ where $v= \begin{pmatrix}
2 & 1
\end{pmatrix}$ yields the matrix \begin{equation*}
\begin{pmatrix*}[r]
 -4& -3\\
 -2& -4\\
\end{pmatrix*} + v^{\top} v = \begin{pmatrix*}[r]
0& -1\\
0& -3\\
\end{pmatrix*},
\end{equation*}
which has trivial Alexander polynomial. We have thus algebraically unknotted $K11a361$ by two moves, showing $u_a(K11a361) \leq 2$.
\end{proof}

Finally, let us consider a generalization of our example $\K$.
\cref{fig:infinite_family} shows an infinite family of knots $\K_n$,
with $\K_1 = \K$. 
\begin{proposition} \label{prop:counterexample_general}
    For $n\geq 1$, let $D_n$ be the positive knot diagram shown in \cref{fig:infinite_family}. Then, the knot $\K_n$ represented by $D_n$ has the following properties.
    \begin{enumerate}
        \item $\K_n$ is fibered and positive. \label{item:fibpos}
        \item $\K_n$ is not braid positive. \label{item:notbraidpos}
        \item $\K_n$ has $3$-and $4$-genus $g(\K_n)=g_4(\K_n)=2+5n$. \label{item:genus}
        \item There is a set of $2+7n$ crossings in the diagram $D_n$ such that changing these crossings turns $D_n$ into a diagram of the unknot. In particular, \label{item:uD}
        $$2+5n\leq u(\K_n)\leq 2+7n.$$
        \item The algebraic unknotting number $u_a(\K_n)$ of $\K_n$ satisfies \label{item:ua}
$$1+2n \leq u_a(\K_n) \leq 7n.$$
        \item The proper rational unknotting number $u_q(\K_n)$ of $\K_n$ satisfies \label{item:uq}
 $$1+2n\leq u_q(\K_n)\leq 1+4n.$$ 
    \end{enumerate}
\end{proposition}
The proposition motivates the following conjecture.

\begin{figure}[tb] 
 	\centering
 	\includegraphics[width=0.7\textwidth]{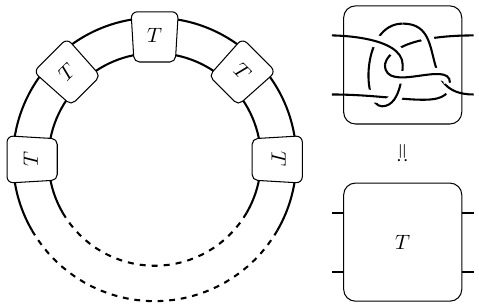}
 	\caption{Diagram $D_n$ representing $\K_n$ is obtained by cyclically connecting $2n+1$ copies of the tangle $T$, as indicated. $D_1$ is shown in \cref{fig:counterexample}.}
\label{fig:infinite_family} 
\end{figure}
\begin{figure}[p] 
 	\centering
 	\includegraphics[width=0.8\textwidth]{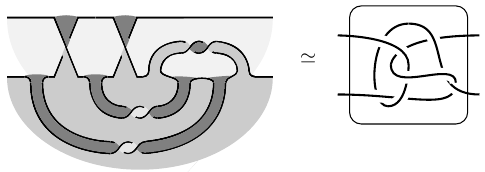}
 	\caption{The fiber surface of $\K_n$ is obtained from $T(2,4n+2)$ by plumbing $3$ positive Hopf bands for each of the copies of $T$.}
\label{fig:counterexampletangle}
\end{figure}
\begin{figure}[p] 
 	\centering
 	\includegraphics[width=0.7\textwidth]{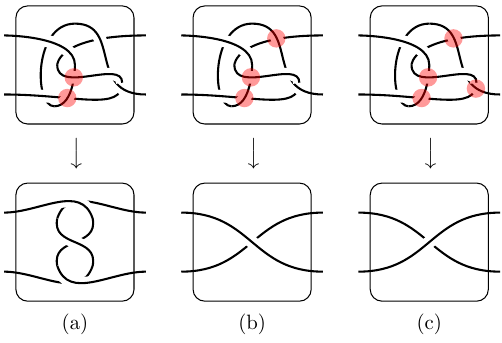}
 	\caption{Changing the marked crossings of the tangle $T$ in the upper row, yields tangles isotopic to the tangles shown in the lower row.
  }
  \label{fig:tangle_unknotting} 
 \end{figure}
 \begin{figure}[p] 
 	\centering
 	\includegraphics[width=0.9\textwidth]{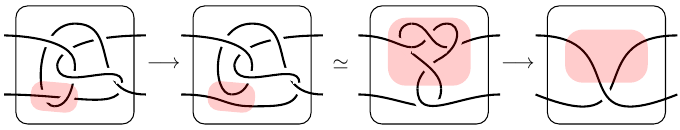}
 	\caption{Two proper rational tangle replacements transform the tangle $T$ into a single positive crossing.}
  \label{fig:untangling_tangle} 
 \end{figure}

 \begin{figure}[p] 
 	\centering
 	\includegraphics[width=\textwidth]{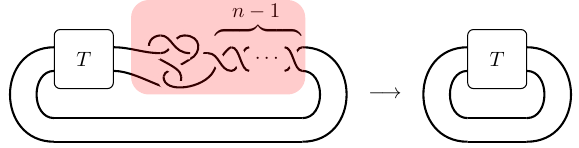}
 	\caption{A proper rational tangle replacement.}
  \label{fig:rational_unknotting} 
 \end{figure}
\begin{conjecture}\label{conj:general}
    The unknotting number of $\K_n$ is $2+7n$. 
\end{conjecture}
Note that this would imply that
none of the $\K_n$ can be obtained from the unknot by plumbing trefoils,
and that all of them are counterexamples to \cref{conj:pos_fib}.
It would even imply that the difference between the unknotting number and the $3$-genus can be arbitrarily large for fibered positive knots (which is not even known for general knots).
As shown in \eqref{item:genus} and \eqref{item:uq},
$g_4(\K_n), u_q(\K_n) \leq g(\K_n)$, so no lower bound for $u$ that is also a lower bound
for $g_4$ or $u_q$ (compare to \cref{rem:types}) can be used to establish \cref{conj:general}.
On the other hand, the best upper bound we found for $u_a(\K_n)$ with $n > 1$ is $7n$,
so a possible strategy to show that $u(\K_n) > g(\K_n)$ for $n > 1$ would be to bound $u_a(\K_n)$ from below.
\begin{proof}[Proof of \cref{prop:counterexample_general}]
    \eqref{item:fibpos} All crossings in $D_n$ are positive and thus $\K_n$ is positive. To see that $\K_n$ is fibered, 
    observe that $\K_n$ is obtained from the unknot by iterated Hopf plumbing. In particular, a Seifert surface for $\K_n$ can be obtained from the fiber surface of the torus link $T(2,4n+2)$ by plumbing $6n+3$ positive Hopf bands, three for each of the $2n+1$ copies of $T$, as shown in \cref{fig:counterexampletangle}.

   \eqref{item:notbraidpos} To see that $\K_1$ is not braid positive, we compute in~\cite{data} Ito's normalized version of the HOMFLYPT polynomial $\tilde{P}_{\K_1}(\alpha,z)$ and observe that some of its coefficients are negative, failing a condition for braid positive knots~\cite[Theorem~1.1]{Ito}. In particular, we observe that $\tilde{P}_{\K_1}(\alpha,0)=-\alpha^4+7\alpha^3+9\alpha^2$. Applying the skein relation for the HOMFLYPT polynomial inductively to the $3n$ and $4n$ crossings involved in the crossing changes in \cref{fig:tangle_unknotting}(b) and (c), we can 
   inductively reduce the general case to the case $n=1$. In particular, it follows that $$\tilde{P}_{\K_n}(\alpha,0)=-n\alpha^{2n+2}+(2+5n)\alpha^{2n+1}+(3+6n)\alpha^{2n},$$ which has a negative coefficient and thus $\K_n$ is not braid positive.

    \eqref{item:genus} The genus of the surface obtained by applying Seifert's algorithm to $D_n$ equals $2+5n$. Since $D_n$ is positive, 
    this surface realizes the $3$- and $4$-genus of $\K_n$. 

    \eqref{item:uD} The lower bound for the unknotting number is given by the $4$-genus of $\K_n$. To get the upper bound, we describe an explicit unknotting sequence in $D_n$. The tangle $T$ can be transformed into a tangle consisting of a single positive (respectively negative) crossing by performing three (resp. four) crossing changes, as shown in \cref{fig:tangle_unknotting}(b) (resp. in \cref{fig:tangle_unknotting}(c)). We perform each of these two transformations in $n$ of the $2n+1$ copies of $T$ in $\K_n$, thus replacing $2n$ copies of $T$ by $n$ positive and $n$ negative crossings. In the remaining tangle, we perform two crossing changes, as shown in \cref{fig:tangle_unknotting}(a). The resulting diagram represents the unknot and is related to $D_n$ by a sequence of $2+3n+4n$ crossing changes. 

  The lower bounds in
\eqref{item:ua} and \eqref{item:uq} are given by the Montesinos trick: the double branched cover $\Sigma_2(K)$ of $K$ can be obtained from $S^3$ by
surgery on a link with $u_q(K)$ components~\cite{McCoy_Zentner}.
Moreover, $\Sigma_2(K)$ can be obtained from an integer homology sphere
by surgery on a link with $u_a(K)$ components (since a knot with
Alexander polynomial~1 has a $\Z HS^3$ as double branched covering).
  In particular, the minimal numbers of required generators of $H_1(\Sigma_2(K); \Z)$ is a lower bound for $u_a(K)$ and~$u_q(K)$. To compute the first homology of the double branched cover we recall that the diagram $D_n$ consists of the cyclic composition of $2n+1$ copies of the tangle $T$. From that, it is not hard to see that the double branched cover of $\K_n$ is obtained by taking $S^1\times (S^2\setminus \bigcup_{2n+1} \mathring D^2)$ and gluing $2n+1$ copies of the double branched cover of $T$ to it. On the other hand, the double branched cover of $T$ can be seen to be the figure eight knot sister $m003$, a $1$-cusped hyperbolic manifold realizing the minimum of volumes among such manifolds. (The other minimizer is the figure eight knot complement.) $m003$ can be seen as the complement of a nullhomologous knot in the lens space $L(5,1)$ and thus has homology $\Z\oplus\Z_5$. In SnapPy~\cite{SnapPy}, \verb|Manifold('L5a1(-5,1)(0,0)').identify()| shows that the $(-5)$-surgery on one of the components
of the Whitehead link $L5a1$ is diffeomorphic to $m003$. (From that surgery description all the above statements are easy to verify.) Now a simple application of the Mayer--Vietoris sequence yields 
    $$ H_1(\Sigma_2(\K_n); \Z)=\left\{ \begin{array}{ll}
\Z_5^{2n+1}\oplus \Z_{2n+1} & \text{if}\ \ 2n+1\equiv 0 \mod 5,\\[0.5em]
\Z_5^{2n}\oplus \Z_{5(2n+1)} & \text{otherwise.} \end{array} \right.
    $$
    In particular, the number of generators in $H_1(\Sigma_2(\K_n);\Z)$ is either $2n+1$ or $2n+2$ depending on the divisibility of $2n+1$. 

    The upper bound for $u_a(\K_n)$ and $u_q(\K_n)$ are obtained by explicit construction.
Performing the crossing changes shown in \cref{fig:tangle_unknotting} (b) and (c) each in $n-1$ copies of $T$,
 yields the knot $\K_1$, for which $u_a(\K_1) \leq 7$ has already been established in \cref{prop:g4uauqK}. Overall, this gives $u_a(\K_n) \leq 3(n-1) + 4(n-1) + 7 = 7n$ as desired.
    \cref{fig:untangling_tangle} shows a sequence of two proper rational tangle replacements transforming $T$ into a single positive crossing. We apply these to $2n-1$ of the tangles $T$ in $D_n$, and to one of the remaining tangles $T$ we just apply the first proper rational tangle replacement from \cref{fig:untangling_tangle}. This yields the knot shown in \cref{fig:rational_unknotting} on the left. The proper rational unknotting replacement shown in that figure yields the closure of $T$, which is the torus knot $T(2,5)$ with proper rational unknotting number $u_q(T(2,5))=1$. In total, we can perform $4n+1$ proper rational tangle replacements to get the unknot, getting the claimed upper bound.
\end{proof}

\let\MRhref\undefined
\bibliographystyle{hamsalpha_mod}
\bibliography{ref.bib}

\newcommand{\etalchar}[1]{$^{#1}$}
\providecommand{\bysame}{\leavevmode\hbox to3em{\hrulefill}\thinspace}
\providecommand{\MR}{\relax\ifhmode\unskip\space\fi MR }
\providecommand{\MRhref}[2]{%
  \href{http://www.ams.org/mathscinet-getitem?mr=#1}{#2}
}
\providecommand{\href}[2]{#2}
\begin{thebibliography}{CDGW}

\bibitem[BD16]{BaaderDehornoy16}
Baader, S. and Dehornoy, P., \emph{Trefoil plumbing}, Proc. Amer. Math. Soc.
  \textbf{144} (2016), 387--397. \MR{3415605}

\bibitem[BF]{knotorious}
Borodzik, M. and Friedl, S., \emph{Knotorious world wide web page},
  \url{http://www.mimuw.edu.pl/~mcboro/knotorious.php}.

\bibitem[BF14]{BFII_ua}
\bysame, \emph{The unknotting number and classical invariants {II}}, Glasg.
  Math. J. \textbf{56} (2014), 657--680. \MR{3250270}

\bibitem[BF15]{BFI_ua}
\bysame, \emph{The unknotting number and classical invariants, {I}}, Algebr.
  Geom. Topol. \textbf{15} (2015), 85--135. \MR{3325733}

\bibitem[BZH14]{MR3156509}
Burde, G., Zieschang, H., and Heusener, M., \emph{Knots}, extended ed., De
  Gruyter Studies in Mathematics, vol.~5, De Gruyter, Berlin, 2014.
  \MR{3156509}

\bibitem[Cro89]{Cromwell_homogeneous}
Cromwell, P.~R., \emph{Homogeneous links}, J. London Math. Soc. (2) \textbf{39}
  (1989), 535--552. \MR{1002465}

\bibitem[CDGW]{SnapPy}
Culler, M., Dunfield, N.~M., Goerner, M., and Weeks, J.~R., \emph{Snap{P}y, a
  computer program for studying the geometry and topology of $3$-manifolds},
  \url{http://snappy.computop.org}.

\bibitem[EL83]{EdmondsLivingston_Fibered}
Edmonds, A.~L. and Livingston, C., \emph{Group actions on fibered
  three-manifolds}, Comment. Math. Helv. \textbf{58} (1983), 529--542.
  \MR{728451}

\bibitem[Eft22]{arXiv:2203.09319}
Eftekhary, E., \emph{Rational tangle replacements and knot {F}loer homology},
  2022, \href{http://arxiv.org/abs/2203.09319}{arXiv:2203.09319}.

\bibitem[Fog93]{Fo_ua}
Fogel, M.~E., \emph{The algebraic unknotting number}, ProQuest LLC, Ann Arbor,
  MI, 1993, Thesis (Ph.D.)--University of California, Berkeley. \MR{2690205}

\bibitem[Gir02]{giroux_hopf}
Giroux, E., \emph{Contact geometry: from dimension three to higher dimensions},
  Proceedings of the {I}nternational {C}ongress of {M}athematicians, {V}ol.
  {II} ({B}eijing, 2002), Higher Ed. Press, Beijing, 2002, pp.~405--414.
  \MR{1957051}

\bibitem[GG06]{GirouxGoodman_plumbing}
Giroux, E. and Goodman, N.~D., \emph{On the stable equivalence of open books in
  three-manifolds}, Geom. Topol. \textbf{10} (2006), 97--114. \MR{2207791}

\bibitem[Goo03]{Goodmann_hopf}
Goodman, N.~D., \emph{Contact structures and open books}, ProQuest LLC, Ann
  Arbor, MI, 2003, Thesis (Ph.D.)--The University of Texas at Austin.
  \MR{2705496}

\bibitem[GL06]{GL_u1}
Gordon, C.~{\relax McA}. and Luecke, J., \emph{Knots with unknotting number 1
  and essential {C}onway spheres}, Algebr. Geom. Topol. \textbf{6} (2006),
  2051--2116. \MR{2263059}

\bibitem[Gre14]{Greene_u1}
Greene, J.~E., \emph{Donaldson's theorem, {H}eegaard {F}loer homology, and
  knots with unknotting number one}, Adv. Math. \textbf{255} (2014), 672--705.
  \MR{3167496}

\bibitem[Har82]{Harer_Hopf82}
Harer, J., \emph{How to construct all fibered knots and links}, Topology
  \textbf{21} (1982), 263--280. \MR{649758}

\bibitem[Hed10]{Hedden_tightSQP}
Hedden, M., \emph{Notions of positivity and the {O}zsv\'{a}th-{S}zab\'{o}
  concordance invariant}, J. Knot Theory Ramifications \textbf{19} (2010),
  617--629. \MR{2646650}

\bibitem[Hir06]{Hironaka}
Hironaka, E., \emph{Salem-{B}oyd sequences and {H}opf plumbing}, Osaka J. Math.
  \textbf{43} (2006), 497--516. \MR{2283407}

\bibitem[ILM21]{iltgen2021khovanov}
Iltgen, D., Lewark, L., and Marino, L., \emph{Khovanov homology and rational
  unknotting}, 2021, \href{http://arxiv.org/abs/2110.15107}{arXiv:2110.15107}.

\bibitem[Ito22]{Ito}
Ito, T., \emph{A note on {HOMFLY} polynomial of positive braid links},
  Internat. J. Math. \textbf{33} (2022), 2250031. \MR{4402791}

\bibitem[KLM{\etalchar{+}}]{data}
Kegel, M., Lewark, L., Manikandan, N., Misev, F., Mousseau, L., and Silvero,
  M., \emph{Code and data to accompany this paper}, Accessible as ancillary
  file at the arxiv version of this paper and at
  \url{https://www.mathematik.hu-berlin.de/~kegemarc/unknotting/counterexample.html}.

\bibitem[Lev69]{levine-signature}
Levine, J., \emph{Knot cobordism groups in codimension two}, Comment. Math.
  Helv. \textbf{44} (1969), 229--244. \MR{246314}

\bibitem[Lin96]{Lines}
Lines, D., \emph{Knots with unknotting number one and generalised {C}asson
  invariant}, J. Knot Theory Ramifications \textbf{5} (1996), 87--100.
  \MR{1373812}

\bibitem[LM23]{KnotInfo}
Livingston, C. and Moore, A.~H., \emph{Knot{I}nfo: Table of knot invariants},
  2023, \url{https://knotinfo.math.indiana.edu}.

\bibitem[McC15]{McCoy_rat_un}
McCoy, D., \emph{Non-integer surgery and branched double covers of alternating
  knots}, J. Lond. Math. Soc. (2) \textbf{92} (2015), 311--337. \MR{3404026}

\bibitem[McC17]{McCoy_u1}
\bysame, \emph{Alternating knots with unknotting number one}, Adv. Math.
  \textbf{305} (2017), 757--802. \MR{3570147}

\bibitem[MZ23]{McCoy_Zentner}
McCoy, D. and Zentner, R., \emph{The {M}ontesinos trick for proper rational
  tangle replacement}, Proc. Amer. Math. Soc. \textbf{151} (2023), 1811--1822.
  \MR{4550372}

\bibitem[MM86]{MR0859157}
Melvin, P.~M. and Morton, H.~R., \emph{Fibred knots of genus {$2$} formed by
  plumbing {H}opf bands}, J. London Math. Soc. (2) \textbf{34} (1986),
  159--168. \MR{859157}

\bibitem[Miy98]{Mi_u1}
Miyazawa, Y., \emph{The {J}ones polynomial of an unknotting number one knot},
  Topology Appl. \textbf{83} (1998), 161--167. \MR{1606374}

\bibitem[Mur85]{Murakami_KnotMetrics}
Murakami, H., \emph{Some metrics on classical knots}, Math. Ann. \textbf{270}
  (1985), 35--45. \MR{769605}

\bibitem[Mur90]{Mu_ua}
\bysame, \emph{Algebraic unknotting operation}, Proceedings of the {S}econd
  {S}oviet-{J}apan {J}oint {S}ymposium of {T}opology ({K}habarovsk, 1989),
  vol.~8, 1990, pp.~283--292. \MR{1043226}

\bibitem[Nak81]{Na_ua}
Nakanishi, Y., \emph{A note on unknotting number}, Math. Sem. Notes Kobe Univ.
  \textbf{9} (1981), 99--108. \MR{634000}

\bibitem[Owe08]{Owens08}
Owens, B., \emph{Unknotting information from {H}eegaard {F}loer homology}, Adv.
  Math. \textbf{217} (2008), 2353--2376. \MR{2388097}

\bibitem[OSS17]{MR3667589}
Ozsv\'{a}th, P.~S., Stipsicz, A.~I., and Szab\'{o}, Z., \emph{Concordance
  homomorphisms from knot {F}loer homology}, Adv. Math. \textbf{315} (2017),
  366--426. \MR{3667589}

\bibitem[OS03]{tau}
Ozsv\'{a}th, P.~S. and Szab\'{o}, Z., \emph{Knot {F}loer homology and the
  four-ball genus}, Geom. Topol. \textbf{7} (2003), 615--639. \MR{2026543}

\bibitem[OS05]{OS_u1}
\bysame, \emph{Knots with unknotting number one and {H}eegaard {F}loer
  homology}, Topology \textbf{44} (2005), 705--745. \MR{2136532}

\bibitem[Ras10]{s_inv}
Rasmussen, J., \emph{Khovanov homology and the slice genus}, Invent. Math.
  \textbf{182} (2010), 419--447. \MR{2729272}

\bibitem[Rud83]{Rudolph83}
Rudolph, L., \emph{Braided surfaces and {S}eifert ribbons for closed braids},
  Comment. Math. Helv. \textbf{58} (1983), 1--37. \MR{699004}

\bibitem[Rud93]{Rudolph93}
\bysame, \emph{Quasipositivity as an obstruction to sliceness}, Bull. Amer.
  Math. Soc. (N.S.) \textbf{29} (1993), 51--59. \MR{1193540}

\bibitem[Rud98]{Ru_plumbing_QP}
\bysame, \emph{Quasipositive plumbing (constructions of quasipositive knots and
  links. {V})}, Proc. Amer. Math. Soc. \textbf{126} (1998), 257--267.
  \MR{1452826}

\bibitem[Rud99]{Rud_Pos_is_SQP}
\bysame, \emph{Positive links are strongly quasipositive}, Proceedings of the
  {K}irbyfest ({B}erkeley, {CA}, 1998), Geom. Topol. Monogr., vol.~2, Geom.
  Topol. Publ., Coventry, 1999, pp.~555--562. \MR{1734423}

\bibitem[Rud05]{Rudolph_ComplexCurves}
\bysame, \emph{Knot theory of complex plane curves}, Handbook of knot theory,
  Elsevier B. V., Amsterdam, 2005, pp.~349--427. \MR{2179266}

\bibitem[Sae99]{Sa_ua}
Saeki, O., \emph{On algebraic unknotting numbers of knots}, Tokyo J. Math.
  \textbf{22} (1999), 425--443. \MR{1727885}

\bibitem[Sch85]{MR0808108}
Scharlemann, M.~G., \emph{Unknotting number one knots are prime}, Invent. Math.
  \textbf{82} (1985), 37--55. \MR{808108}

\bibitem[Sta78]{Stallings78}
Stallings, J.~R., \emph{Constructions of fibred knots and links}, Algebraic and
  geometric topology ({P}roc. {S}ympos. {P}ure {M}ath., {S}tanford {U}niv.,
  {S}tanford, {C}alif., 1976), {P}art 2, Proc. Sympos. Pure Math., vol. XXXII,
  Amer. Math. Soc., Providence, RI, 1978, pp.~55--60. \MR{520522}

\bibitem[Sto03]{Stoimenow03}
Stoimenow, A., \emph{Positive knots, closed braids and the {J}ones polynomial},
  Ann. Sc. Norm. Super. Pisa Cl. Sci. (5) \textbf{2} (2003), 237--285.
  \MR{2004964}

\bibitem[Sto04]{St004}
\bysame, \emph{Polynomial values, the linking form and unknotting numbers},
  Math. Res. Lett. \textbf{11} (2004), 755--769. \MR{2106240}

\bibitem[Sag23]{sage}
{The Sage{~}Developers}, \emph{{S}agemath, the {S}age {M}athematics {S}oftware
  {S}ystem ({V}ersion 10.1)}, 2023, \url{http://www.sagemath.org}.

\bibitem[Tri69]{tristram}
Tristram, A.~G., \emph{Some cobordism invariants for links}, Proc. Cambridge
  Philos. Soc. \textbf{66} (1969), 251--264. \MR{248854}

\bibitem[Was97]{MR1421575}
Washington, L.~C., \emph{Introduction to cyclotomic fields}, second ed.,
  Graduate Texts in Mathematics, vol.~83, Springer-Verlag, New York, 1997.
  \MR{1421575}

\bibitem[Wen37]{wendt}
Wendt, H., \emph{Die gordische {A}ufl{\"o}sung von {K}noten}, Math. Z.
  \textbf{42} (1937), 680--696. \MR{1545700}

\end{thebibliography}
\end{document}